\documentclass[12pt, reqno]{amsart}
\usepackage[cp1251]{inputenc}
\usepackage{amsmath,amsopn,amssymb,amsthm, wasysym}
\usepackage[breaklinks=true,colorlinks=true,linkcolor=blue,citecolor=blue,urlcolor=blue]{hyperref}
\usepackage[english]{babel}
\usepackage{graphicx}

\renewenvironment{proof}[1][Proof]{\textbf{#1.} }{\ \rule{0.5em}{0.5em}}

\DeclareMathOperator{\Isom}{Isom}

\DeclareMathOperator{\Lin}{Lin}
\DeclareMathOperator{\Symm}{Symm}

\DeclareMathOperator{\I}{I}

\renewenvironment{proof}[1][Proof]{\textbf{#1.} }
{\ \rule{0.5em}{0.5em}}

\newtheorem{theorem}{Theorem}
\newtheorem{prop}{Proposition}
\newtheorem{lemma}{Lemma}
\newtheorem{corollary}{Corollary}
\newtheorem{quest}{Question}
\newtheorem{problem}{Problem}

\theoremstyle{definition}
\newtheorem{definition}{Definition}
\newtheorem{remark}{Remark}
\newtheorem{example}{Example}

\begin{document}
\selectlanguage{english}

\title
[On $m$-point homogeneous polytopes in Euclidean spaces]
{On $m$-point homogeneous polytopes \\in Euclidean spaces}
\author{V.N.~Berestovski\u\i, Yu.G.~Nikonorov}

\address{Berestovski\u\i\  Valeri\u\i\  Nikolaevich \newline
Sobolev Institute of Mathematics of the SB RAS, \newline
4 Acad. Koptyug Ave., Novosibirsk 630090, RUSSIA, \newline
Principal scientific researcher}
\email{valeraberestovskii@gmail.com}

\address{Nikonorov\ Yuri\u\i\  Gennadievich\newline
Southern Mathematical Institute of VSC RAS \newline
53 Vatutina St., Vladikavkaz, 362025, RUSSIA, \newline
Principal scientific researcher}
\email{nikonorov2006@mail.ru}

\thanks{The work of the first author was carried out within the framework of the state Contract to the IM SB RAS, project FWNF-2022-0006.}

\begin{abstract}
This paper is devoted to the study the $m$-point homogeneity property and the point homogeneity degree for finite metric spaces.
Since the vertex sets of regular polytopes, as well as of some their generalizations, are homogeneous, we
pay much attention to the study of the homogeneity properties of the vertex sets of polytopes in Euclidean spaces.
Among main results, there is a classification of polyhedra with all edges of
equal length and with 2-point homogeneous vertex sets.
In addition, a significant part of the paper is devoted to the development of methods and tools
for studying the relevant objects.

\vspace{2mm}
\noindent
2020 Mathematical Subject Classification: 54E35, 52B15, 20B05.

\vspace{2mm} \noindent Key words and phrases:
Archimedean solid,
finite homogeneous metric space, Gosset polytope, $m$-point homogeneous metric space,
 Platonic solid, point homogeneity degree, regular polytope,
semiregular polytope.
\end{abstract}

\maketitle

\section{Introduction}\label{sec.0}

The main object of our study are homogeneous metric spaces with special properties.
For instance, we interested in {\it homogeneous} finite metric spaces $(M,d)$ which means that the isometry group $\Isom(M)$ of $(M,d)$ acts transitively on
$M$.
A finite homogeneous metric subspace of an Euclidean space
represents the vertex set of a compact convex polytope
with the isometry group that is transitive on the vertex set;
in each case, all vertices lie on a sphere.
In \cite{BerNik19, BerNik21, BerNik21n}, the authors  obtained the complete description of the metric properties
of the vertex sets of regular and semiregular polytopes
in Euclidean spaces from the point of view of the normal homogeneity and the Clifford~---~Wolf homogeneity, see also survey \cite{BerNik21nn}.
The coresponding classes of (generalized) normal homogeneous and
Clifford~--- Wolf homogeneous  Riemannian manifolds were studied earlier
in~\cite{BerNik2008, BerNik08, BerNik2009,  BerNik14, BerNik20}.

\medskip

This paper is devoted to another one special property of metric spaces, which is a strengthening of the notion of homogeneity.

\begin{definition}\label{de:mpoint}
A metric space $(M,d)$ is called {\it $m$-point homogeneous}, $m \in \mathbb{N}$, if for every $m$-tuples
$(A_1, A_2,\dots, A_m)$ and $(B_1,B_2,\dots, B_m)$  of elements of $M$ such that $d(A_i,A_j)=d(B_i,B_i)$, $i,j =1,\dots, m$,
there is an isometry $f \in \Isom(M)$  with the following property: $f(A_i)=B_i$, $i=1,\dots, m$.
\end{definition}

From a formal point of view, there can be coinciding points among the points
$A_1,\dots,A_m$ (as well as among the points $B_1,\dots,B_m$) in the above definition.
This observation immediately implies the following result.

\begin{prop}\label{pr:m-point hom_obv}
If a metric space $(M,d)$ is $m$-point homogeneous for some $m \geq 2$, then it is $k$-point homogeneous for any  $k=1,\dots, m-1$.
If the set $M$ is finite, $l$ is its cardinality, and $(M,d)$ is $l$-point homogeneous, then $(M,d)$ is $m$-point homogeneous for all $m \geq 1$.
\end{prop}

\medskip

The most interesting results of this paper related to finite subsets of Euclidean space~$\mathbb{R}^n$ with some degree of homogeneity.
It is assumed that any such set $M$ is supplied with the metric $d$
induced from $\mathbb{R}^n$, hence, is a finite metric space itself.

Since the barycenter of a finite system of material points (with one and the same mass) in any Euclidean space is preserved for any bijection
(in particular, any isometry)
of this system, we have the following result.

\begin{prop}[\cite{BerNik19}] \label{pr.efhs}
Let $M=\{x_1, \dots, x_q\}$, $q\geq n+1$, be a finite homogeneous metric subspace of Euclidean
space $\mathbb{R}^n$, $n\geq 2$. Then
$M$ is the vertex set of  a convex polytope~$P$, that is situated in some sphere in $\mathbb{R}^n$ with radius
$r>0$ and center $x_0=\frac{1}{q}\cdot\sum_{k=1}^{q}x_k$.
In particular, $\Isom(M,d)\subset O(n)$.
\end{prop}

This result shows that {\it the theory of convex polytopes} is very important for the study of finite homogeneous subspaces of Euclidean spaces.
For a more detailed acquaintance with the theory of convex polytopes, we recommend \cite{Berg84, Cox73, Crom97, Grun}.
\smallskip

Let $P$ be a non-degenerate convex polytope in $\mathbb{R}^n$  with the barycenter in the origin $O=(0,0,\dots,0)\in \mathbb{R}^n$
(the property to be non-degenerate means that $\Lin(P)=\mathbb{R}^n$ or, equivalently, $O$ is in the interior of $P$).
The symmetry group $\Symm(P)$ of $P$ is the group of isometries of $\mathbb{R}^n$ that preserve $P$. It is clear that
each $\psi \in \Symm(P)$ is an orthogonal transformation of $\mathbb{R}^n$ (obviously, $\psi(O)=O$ for any symmetry $\psi$ of~$P$).

If $M$ is the vertex set of a polytope $P$ supplied with the metric $d$ induced from the Euclidean one in $\mathbb{R}^n$,
then the isometry group $\Isom(M,d)$ of the metric space $(M,d)$ is the same as $\Symm(P)$.

Recall that a polytope~$P$ in~$\mathbb{R}^n$
is {\it homogeneous\/} (or {\it vertex-transitive}) if
its symmetry (isometry) group acts transitively on the set of its vertices.
The following definition is natural.

\begin{definition}\label{de_poly_hom}
A convex polytope $P$ in $\mathbb{R}^n$ is called $m$-point homogeneous if its vertex set {\rm(}with induced metric $d$ from $\mathbb{R}^n${\rm)}
is $m$-point homogeneous.
\end{definition}

In particular, the property to be $1$-point homogeneous means just the property to be homogeneous.
Recall that regular as well as semiregular polytopes in Euclidean spaces are homogeneous (see details in Section \ref{sec.2}).

There are many equivalent definition of regular (convex) polytopes \cite{Mar}.
One of this is as follows: A polytope $P$ is regular if its symmetry group acts transitively on flags of $P$.
Recall that {\it a flag} in an $n$-dimensional polytope $P$ is a sequence of its faces $F_0, F_1, \dots, F_{n-1}$, where
$F_k$ is an $k$-dimensional face and $F_i \subset F_{i+1}$, $i=0,1,\dots,n-2$.
See discussion of this property e.~g. in \cite[Chapter 4]{DuVal}.

\smallskip

The above definition leads to the conjecture that
{\it the vertex set of any $n$-dimensional regular polytope is $n$-point homogeneous}.
But this conjecture is false (see details in Section \ref{se_m-point}).
Therefore, the following problem is natural:

\begin{problem}\label{pr:kpoin.euclid}
Classify all convex polytopes $P$ in $\mathbb{R}^n$ whose vertex sets are $m$-point homogeneous, where $m \geq 2$.
\end{problem}

One can check that every regular polyhedron  in $\mathbb{R}^3$ is $2$-point homogeneous (Proposition~\ref{pr:3dim.2point.reg}).
The same property has the cuboctahedron (Proposition~\ref{pr:co}).
The main results of this paper are the following theorems.

\begin{theorem}\label{th:reg_pol}
The tetrahedron, cube, octahedron, icosahedron are $m$-point homogeneous polyhedra for every natural $m$, while
the dodecahedron is $2$-point homogeneous but is not $3$-point homogeneous.
\end{theorem}

\begin{theorem}\label{th:edge}
If all edges of a given $2$-point homogeneous polyhedra $P\subset \mathbb{R}^3$ have the same length, then $P$ is either the cuboctahedron or a regular polyhedron.
\end{theorem}

We obtain also several interesting results on $m$-point homogeneous finite subsets in Euclidean spaces $\mathbb{R}^n$ for an arbitrary $n$.
In particular, we get a couple of results on the $m$-point homogeneity degree
(see Definition \ref{de_degree_hom}) for some particular types of polytopes.
In order to do this, we developed some special tools and approaches that could be useful in the future studies.
\smallskip

As a rule, we use standard notation. When we deal with some metric space $(M,d)$,
for given $c\in M$ and $r \geq 0$, we consider $S(c,r)=\{x\in M\,|\, d(x,c)=r\}$,  $U(c,r)=\{x\in M\,|\, d(x,c)< r\}$,
and  $B(c,r)=\{x\in M\,|\, d(x,r)\leq r\}$, respectively
the sphere, the open ball, and the closed ball with the center $c$ and radius $r$. For $x,y \in \mathbb{R}^n$, the symbol $[x,y]$ means the closed interval in $\mathbb{R}^n$
with the ends $x$ and $y$.
\smallskip

The paper is organized as follows. In Section \ref{sec.1} we give some  important information on regular and semiregular polytopes in Euclidean spaces.
In Section \ref{sec.2} we prove some important results of general nature.
In Section \ref{sec.3} we study the $m$-point homogeneity property for some finite homogeneous metric spaces, mainly
the vertex sets of regular and semiregular polytopes.
In Section \ref{se_m-point} we obtain some results on the point homogeneity degree for some important classes of polytopes.
Section \ref{se_m-point.3-dim} is devoted to the study of $2$-point homogeneous polyhedra with equal edge lengths.
In particular, we prove Theorems \ref{th:reg_pol} and  \ref{th:edge}.
Finally, in Section \ref{ne_m-point.3-dim}, we study the $m$-point homogeneity for polyhedra with several edge lengths.
In Conclusion, some information on $2$-point-homogeneous Riemannian manifolds and some of their generalizations is given.
\smallskip

\section{Regular and semiregular polytopes}\label{sec.1}

In this section, we recall some important properties of regular and semiregular polytopes in Euclidean spaces.
Faces of dimension $n-1$ (hyperfaces) of a $n$-dimensional polytope
commonly referred to as {\it facets}.
Note that they are also called {\it cells\/} for $n=4$.

A $P$ is called a {\it polytope with~regular faces\/}
(respectively, a {\it polytope with congruent faces\/}),
if all its facets are regular
(respectively, congruent) polytopes.

A one-dimensional polytope is a closed segment,
bounded by two endpoints.
It is regular by definition.
Two-dimensional regular polytopes are convex regular
polygons on Euclidean plane.
For other dimensions, regular polytopes are defined inductively.
A convex $n$-dimensional polytope for $n \geq 3$ is called {\it regular},
if it is homogeneous
and all its facets are regular congruent to each other
polytopes.
This definition is equivalent to other definitions of regular convex
polytopes (see~\cite{Mar}).
A {\it vertex figure\/} of $n$-dimensional regular polytope,
$n\geq 3$, is a $(n-1)$-dimensional polytope,
which is the convex hull of the vertices,
having a common edge with a given vertex and different from it.

We also recall the definition of a wider class of semiregular convex polytopes.
For $n=1$ and $n=2$, semiregular polytopes are defined as regular.
A convex $n$-dimensional polytope for $n\geq 3$ is called
{\it semiregular} if it is homogeneous
and all its facets are regular polytopes.

A generalization of the class of semiregular polytopes is the class of uniform polytopes.
For $n\leq 2$, uniform polytopes are defined as regular.
For other dimensions, uniform polytopes are defined inductively.
A convex $n$-dimensional polytope for $n\geq 3$ is called
{\it uniform} if it is homogeneous
and all its facets are uniform polytopes.
In particular, for $n=3$, the classes of uniform and semiregular polytopes
coincide,
and for $n=4$ the facets of the uniform polytope must be
semiregular three-dimensional polytopes.
This class of polytopes is far from complete classification,
see known results in~\cite{{4D},{Mar}}.

The classification of regular polytopes of arbitrary dimension was first obtained by Ludwig Schl\"{a}fli and is presented in his book \cite{Schl},
see also Harold Coxeter's book~\cite{Cox73}.
The list of semiregular polytopes of arbitrary
dimension was first presented without proof in Thorold~Gosset's paper \cite{Gos}.
Later this list appeared in the work of Emanuel~Lodewijk~Elte \cite{Elte12}.
The proof of the completeness
of this list was obtained much later by Gerd~Blind and Rosvita~Blind, see  \cite{BlBl} and the references therein.
Semiregular (non-regular) polytopes in $\mathbb{R}^n$ for $n\geq 4$ are called {\it Gosset polytopes}.
A lot of additional information can be found in~\cite{rsr}.

We briefly recall the classification of regular and semiregular polytopes
in Euclidean spaces.

An one-dimensional polytope (closed segment) is regular.
Two-dimensional regular polyhedra (polygons) have equal sides
and are inscribed in a circle.

For regular three-dimensional polyhedra, the vertex figure
is a regular polygon.
It is well known that there are only five regular three-dimensional
polyhedra:
the tetrahedron, cube, octahedron, dodecahedron, and icosahedron.
These polyhedra are traditionally called {\it Platonic solids\/}.
Some important properties of these polyhedra could be found in~Table\,\ref{table0},
where $V$, $E$, and $F$ mean respectively the numbers of vertices, edges, and faces;
$\alpha$ is the dihedral angle; the number $\varphi:=\frac{1+\sqrt{5}}{2}$ is known as {\it the golden ratio}.

\renewcommand{\arraystretch}{1.4} 
\begin{table}[t]
\caption{
Regular 3-dimensional polyhedra.}
\label{table0}
\begin{center}
\begin{tabular}
{|p{0.3\linewidth}|p{0.05\linewidth}|p{0.05\linewidth}|p{0.05\linewidth}|p{0.27\linewidth}|p{0.09\linewidth}|}
\hline
Polyhedron                      &$V$ &$E$ &$F$ &$\alpha$ & Face  \\
\hline\hline
Tetrahedron                     &4   & 6  & 4  & $2\arcsin(1/\sqrt{3})$                            &{\large $\triangle$}  \\\hline
Cube (hexahedron)               &8   & 12 & 6  & $\pi/2$                                           &{\Large $\square$}    \\\hline
Octahedron                      &6   & 12 & 8  & $2\arcsin(\sqrt{2/3})$                            &{\large $\triangle$}  \\\hline
Dodecahedron                    &20  & 30 & 12 & $2\arcsin\left(\sqrt{\varphi}/\sqrt[4]{5}\right)$ &\Large \pentagon  \\\hline
Icosahedron                     &12  & 30 & 20 & $2\arcsin(\varphi/\sqrt{3})$                      &{\large $\triangle$} \\
\hline
\end{tabular}
\end{center}
\end{table}
\renewcommand{\arraystretch}{1}

There are 6 regular $4$-dimensional polytopes:
the hypertetrahedron ($5$-cell), hypercube  ($8$-cell),  hyperoctahedron ($16$-cell), $24$-cell, $120$-cell, and  $600$-cell.
A detailed description of the structure of four-dimensional regular
polytopes could be found e.~g. in \cite[Section 3]{BerNik21}.

For each dimension $n\geq 5$, there exists
exactly three regular polytopes:
the $n$-dimensional simplex,
the hypercube ($n$-cube) and the hyperoctahedron ($n$-orthoplex).

We recall some information on the structure of these polytopes and their symmetry groups.
In what follows, $S_n$ is denoted the symmetric group of degree $n$, i.~e. is the group of all permutations on $n$ symbols.
Recall that the order of $S_n$ is $n!$. We will use also the cyclic groups $\mathbb{Z}_k$ of order $k$.

Clearly, the symmetry group of the regular $n$-dimensional simplex $P$ is the group~$S_{n+1}$.

The symmetry group of the $n$-cube $P$ is the group $\mathbb{Z}_2^n \rtimes S_n$. It is easy to check if we represent
$P$ as a convex hull of the points $(\pm a,\pm a,\cdots, \pm a)\in \mathbb{R}^n$,
where $a>0$, and the sign could take value $-$ or $+$
on each place independently on other places. Here, $S_n$ is the group of permutation of all coordinates, and the group $\mathbb{Z}_2^n$
is generated by the maps $x\mapsto -x$ for all coordinates.

The symmetry group of the $n$-orthoplex $P$ is also the group $\mathbb{Z}_2^n \rtimes S_n$. It is easy to check if we represent
$P$ as a convex hull of the points $\pm b \, e_i$, $i=1,2,\dots,n$, where $b>0$ and $e_i$ is the $i$-th basic vector in $\mathbb{R}^n$
(it has all zero coordinates excepting $1$ in the $i$-th place).
This result could be obtained in more simple way. Since the $n$-orthoplex is dual to the $n$-cube, they have one and the same symmetry group.
The duality means that the centers of all facets of the $n$-orthoplex (the $n$-cube) constitute the vertex set of the $n$-cube (respectively, the $n$-orthoplex).

We now proceed to a brief description of the semiregular (non-regular) polytopes.

In three-dimensional space (in addition to Platonic solids), there are
the following semiregular polyhedra:
13 {\it Archimedean solids\/} and two infinite series
of regular prisms and right antiprisms. A detailed description of Archimedean solids could be found in Sections 4 and 5 of~\cite{BerNik21}.

A {\it right prism} is a polyhedron whose two faces (called bases) are congruent (equal) regular polygons,
lying in parallel planes, while other faces (called lateral ones) are rectangles (perpendicular to
the bases).
If lateral faces are squares then the prism is said to be {\it regular}.
In this case we get an infinite family of semiregular convex polyhedra.

A {\it right antiprism} is a semiregular polyhedron, whose two parallel faces (bases) are equal regular $n$-gons,
while other $2n$ (lateral) faces are regular triangles.
Note that the octahedron is an antiprism with triangular bases.

For $n=4$, we have exactly three semiregular polytopes:
the rectified 4-simplex,  the rectified 600-cell, and the snub 24-cell.
A detailed description of these polytopes could be found in Section 5 of \cite{BerNik21n}.

The unique (up to similarity)
semiregular Gosset polytope in $\mathbb{R}^n$ for $n\in \{5,6,7,8 \}$ we denote by the symbol $\operatorname {Goss\,}_n$.
Detailed descriptions of these polytopes could be found  in Sections 6,7,8, and 9 of \cite{BerNik21n} respectively.

\begin{example}\label{ex:Gos6.1}
Let us consider a brief explicit description of  $\operatorname {Goss\,}_6$.
This polytope
can be implemented in different ways. Let us set it with the coordinates of the vertices in $\mathbb{R}^6$, as it is done in~\cite{Elte12}.
Let us put $a=\frac{\sqrt{2}}{4}$ and $b=\frac{\sqrt{6}}{12}$. We define the points $A_i \in \mathbb{R}^6$, $i=1,\dots,27$, as follows:

\smallskip
\noindent
{\small
\begin{tabular}{lll}
$A_1=(0,0,0,0,0,4b)$, & $A_2=(a,a,a,a,a,b)$, & $A_3=(-a,-a,a,a,a,b)$,\\$A_4=(-a,a,-a,a,a,b)$, & $A_5=(-a,a,a,-a,a,b)$, & $A_6=(-a,a,a,a,-a,b)$,\\
$A_7=(a,-a,-a,a,a,b)$, & $A_8=(a,-a,a,-a,a,b)$, & $A_9=(a,-a,a,a,-a,b)$,\\
$A_{10}=(a,a,-a,-a,a,b)$, & $A_{11}=(a,a,-a,a,-a,b)$, & $A_{12}=(a,a,a,-a,-a,b)$,\\
$A_{13}=(-a,-a,-a,-a,a,b)$, & $A_{14}=(-a,-a,-a,a,-a,b)$, & $A_{15}=(-a,-a,a,-a,-a,b)$,\\
$A_{16}=(-a,a,-a,-a,-a,b)$, & $A_{17}=(a,-a,-a,-a,-a,b)$, & $A_{18}=(2a,0,0,0,0,-2b)$,\\
$A_{19}=(0,2a,0,0,0,-2b)$, & $A_{20}=(0,0,2a,0,0,-2b)$, & $A_{21}=(0,0,0,2a,0,-2b)$,\\
$A_{22}=(0,0,0,0,2a,-2b)$, & $A_{23}=(-2a,0,0,0,0,-2b)$, & $A_{24}=(0,-2a,0,0,0,-2b)$,\\
$A_{25}=(0,0,-2a,0,0,-2b)$, & $A_{26}=(0,0,0,-2a,0,-2b)$, & $A_{27}=(0,0,0,0,-2a,-2b).$\\
\end{tabular}}
\medskip

The Gosset polytope $\operatorname{Goss\,}_6$ is the convex hull of these points.
It is easy to check that $d(A_1, A_i)=1$ for $2\leq i \leq 17$ and $d(A_1, A_i)=\sqrt{2}$ for $18\leq  i \leq 27$.

It is clear that the points $A_2 - A_{17}$ are vertices of a five-dimensional demihypercube (the corresponding hypercube has $32$ vertices of the form
$(\pm a, \pm a,\pm a,\pm a,\pm a,b)$),
and the points $A_{18} - A_{27}$ are the vertices of the five-dimensional hyperoctahedron (orthoplex),
which is a facet of the polytope $\operatorname{Goss\,}_6$ (lying in the hyperplane $x_6=-2b$).
The origin $O=(0,0,0,0,0,0)\in \mathbb{R}^6$ is the center of the hypersphere described around $\operatorname{Goss\,}_6$ with radius $4b=\sqrt{2/3}$.
\end{example}

\section{General results}\label{sec.2}

Now we will prove one useful result on inscribed polytopes in $\mathbb{R}^n$.

\begin{theorem}\label{th: circ_dist}
Let $P$ be a convex polytope in $\mathbb{R}^n$ such that all its vertices lie in some hypersphere.
Then for every vertex $A$ of $P$, every nearest vertex $B$ of $P$ for $A$
is adjacent to $A$. As a corollary, the minimal distance between distinct vertices of $P$ is attained only on the pairs of adjacent vertices.
\end{theorem}

\begin{proof} Without loss of generality we may suppose that all vertices of $P$ are in $S(O,1)$, where $O$ is the origin in $\mathbb{R}^n$.
Let us fix a vertex $A$ of $P$ and define for $X\in P\setminus \{A\}$
the function
$$
f(X)=2(\overrightarrow{AO}, \overrightarrow{AX})\cdot|\overrightarrow{AX}|^{-1},
$$
which is equal to $2\cos\alpha,$ where $\alpha$ is the angle between vectors
$\overrightarrow{AO}$ and $\overrightarrow{AX}$. It is clear that $f$ is positive. It has yet another geometric sense: $f(X)=|AY|$, where $Y\neq A$
is the intersection point of the ray $r(X):=\{t\cdot \overrightarrow{AX}, t\geq 0\}$ with $S(O,1)$. This is true if $X=Y\in S(O,1)$,
because we have $f(X)=2\cos\alpha$ for the isosceles triangle $\Delta AOX$ with lateral sides $|OA|=|OX|=1$. In another case this is true because $f(X)=f(Y)$.

Since $P$ is convex, then $C(A):=\{r(X), X\in P\setminus \{A\}\}$ is a convex
cone with the vertex $A.$
Let us choose a point $C$ on $(OA)$ such that
the hyperplane $H(C)$, passing through $C$ orthogonally to $\overrightarrow{AO}$, separates the vertex $A$ from other vertices of $P$. It is clear that
$\mathcal{P}=C(A)\cap H(C)$ is a convex $(n-1)$-dimensional polytope, lying
inside closed $(n-1)$-dimensional ball $B^{n-1}(C,r)=B^n(O,1)\cap H(C)$ with the center $C$ and radius $r=\sqrt{1-|OC|^2}.$ There exists the smallest number
$r_0>0$ such that
$$\mathcal{P}\subset B^{n-1}(C,r_0)=B^n(C,r_0)\cap H(C).$$
Then $r_0< r$ and the set $M_0:=\mathcal{P}\cap S^{n-2}(C,r_0),$ where $S^{n-2}(C,r_0)=S^{n-1}(C,r_0)\cap H(C),$  is non-empty and consists only from some
vertices of $\mathcal{P}$. If $X_0$ is any such vertex, then it is clear
from the previous argument that $f(X_0)=\min\{f(X): X\in P\setminus \{A\}\}$,
the point $X_0$ necessarily lies on some edge $e$ of $P$ with the vertex $A$, and the length of $e$ is equal to $f(X_0).$ The above considerations imply
the statements of Theorem \ref{th: circ_dist}.
\end{proof}

\smallskip

We get the following obvious

\begin{corollary}\label{th: homog_dist} Let $P$ be a homogeneous convex polytope in $\mathbb{R}^n$. Then
the minimal distance between distinct vertices of $P$ is attained on some pair of adjacent vertices.
\end{corollary}

\begin{remark}
It is easy to see that for the rhombus in $\mathbb{R}^2$ with the vertices $(-1,0)$, $(1,0)$, $(0,\alpha)$, and $(0,-\alpha)$, the minimal distance
between distinct vertices is attained on the vertices $(0,\alpha)$ and $(0,-\alpha)$ (that are not adjacent) for sufficiently small $\alpha$.
Analogous example could be easily produced in $\mathbb{R}^n$, $n \geq 3$.
This shows that the condition that a polyhedron is inscribed in a hypersphere is necessary in Theorem \ref{th: circ_dist}.
\end{remark}

\smallskip

Let us consider any homogeneous metric space $(M,d)$ with the isometry group $\Isom(M)$.
For any $x\in M$, the group $\I(x)=\{\psi \in \Isom(M) \,|\, \psi(x)=x \}$ is called {\it the isotropy subgroup at the point $x$}.
If $\eta \in \Isom(M)$ is such that $y=\eta (x)$, then, obviously, $\I(y)= \eta \circ \I(x) \circ \eta^{-1}$. Hence, the isotropy subgroups for every two points of
a given homogeneous finite metric $M$ space are conjugate each to other in the isometry group.
In particular, the cardinality of $M$ is the quotient of the cardinality of $\Isom(M)$ by the cardinality of $\I(x)$ (for any $x\in M$).
\smallskip

We will discuss some properties of $2$-point homogeneous metric spaces.
According to Definition \ref{de:mpoint},
a metric space $(M,d)$ is {\it two-point homogeneous {\rm(}$2$-point homogeneous{\rm)}},
if for every pairs  $(A_1,A_2)$ and $(B_1,B_2)$ of elements of $M$ such that $d(A_1,A_2)=d(B_1,B_2)$,
there is an isometry $f \in \Isom(M)$  with the following property: $f(A_i)=B_i$, $i=1,2$.

\begin{remark}
There are some publications, where some types of two-point homogeneity for finite sets with additional structures are studied.
For instance, we refer to \cite{Tam13},
where two-point homogeneous quandles are studied. On the other, we are focused on the study of metric spaces.
\end{remark}

For a given $x\in M$ and $r >0$ we consider the sphere $S(x,r)=\{y\in M\,|\,d(x,y)=r\}$ with the center $x$ and radius $r$.

\begin{remark}
If $(M,d)$ is a homogeneous subset in Euclidean space $\mathbb{R}^n$, then every sphere $S(x,r)$ lies in some $(n-1)$-dimensional Euclidean subspace,
since all points of this sphere are also on one and the same distance from the barycenter of $M \subset \mathbb{R}^n$
(recall that $\|x-a\|=C$ and $\|x-b\|=D$ for $x, a, b \in \mathbb{R}^n$
implies $2(x,a-b)=\|a\|^2-\|b\|^2-C^2+D^2$).
\end{remark}

It is easy to prove the following helpful result.

\begin{prop}\label{pr:two-point homogeneous}
A homogeneous metric space $(M,d)$ is two-point homogeneous if and only if for any point $x\in M$
the following property holds: for every $r >0$ and every $y,z \in S(x,r)$, there is
$f \in \I(x)$ such that $f(y)=z$. In other words, a homogeneous metric space $(M,d)$ is two-point homogeneous if and only if
for any point $x\in M$, the isotropy subgroup $\I(x)$ acts transitively on every sphere $S(x,r)$, $r>0$.
\end{prop}

In particular, Proposition \ref{pr:two-point homogeneous} implies

\begin{prop}\label{pr:3dim.2point.reg}
The vertex set of any regular polytope in $\mathbb{R}^3$ is $2$-point homogeneous.
\end{prop}

The following result on the $m$-point homogeneity could be useful in some  explicit calculations.

\begin{prop}\label{pr:m-point homogeneous}
If a metric space $(M,d)$ is $m$-point homogeneous, $m \geq 2$, then for any point $x\in M$
the following property holds: for every $r >0$ and every
$(m-1)$-tuples
$(A_1, A_2,\dots, A_{m-1})$ and $(B_1,B_2,\dots, B_{m-1})$  of elements of $S(x,r)$ such that $d(A_i,A_j)=d(B_i,B_j)$, $i,j =1,\dots, m-1$, there is
$f \in \I(x)$ such that $f(A_i)=B_i$ for all $i=1,\dots,m-1$. In other words, if a homogeneous metric space $(M,d)$ is $m$-point homogeneous,
then for any point $x\in M$, the isotropy subgroup $\I(x)$ acts $(m-1)$-point transitively on every sphere $S(x,r)$, $r>0$
{\rm(}in particular, $S(x,r)$ with the metric, induced by $d$, is $(m-1)$-point homogeneous{\rm)}.
\end{prop}

\begin{proof}
We can extend every $(m-1)$-tuples
$(A_1, A_2,\dots, A_{m-1})$ and $(B_1,B_2,\dots, B_{m-1})$  of elements of $S(x,r)$ such that $d(A_i,A_j)=d(B_i,B_j)$, $i,j =1,\dots, m-1$,
to $m$-tuples adding the points $A_m=B_m=x$. Since $d(x,A_i)=d(x,B_i)=r$, $i=1,\dots,m-1$, and $(M,d)$ is $m$-point homogeneous,
there is an isometry $f \in \Isom(M)$  with the following property: $f(A_i)=B_i$, $i=1,\dots, m$. Since $A_m=B_m=x$ then $f \in \I(x)$, hence, $f$ an isometry of $S(x,r)$.
\end{proof}

\begin{remark}
It should be noted that it could be some isometries of a given sphere $S(x,r)$ that are not generated by the elements of the isotropy subgroup $\I(x)$.
\end{remark}

\begin{lemma}
\label{inversion}
Let $P$ be a $n$-dimensional convex centrally symmetric polytope with the vertex set $M$ inscribed in the sphere $S^{n-1}(O,r)$
with the center in the origin of $\mathbb{R}^n$, $M_1$, $M_2$ be some ordered subsets in $M$, for which there exists an isometry $f$
of the polytope $P$ and the sphere $S(O,r)$ on themselves such that $f(M_1)=M_2$, $v\in M_1$, $w\in M_2$, $f(v)=w$. Then
$$
f(M_1\cup \{-v\})=    M_2\cup \{-w\},\quad f[(M_1-\{v\})\cup \{-v\}]=(M_2-\{w\})\cup \{-w\}.
$$
\end{lemma}

\begin{proof}
Clearly, $f(-v)=-f(v)=-w.$ It implies the first equality of the lemma. The second equality is a consequence of the first one.
\end{proof}

\begin{theorem}\label{th:3dist_symm}
Let $P$ be a polytope in $\mathbb{R}^n$, $n\geq 2$, such that its vertex set $M$ is homogeneous,
symmetric with respect to the center of $P$, and the distances between distinct vertices of $P$
constitute a 3-element set, say $\{d_1, d_2, d_3\}$, where $0<d_1<d_2<d_3$.
If for some vertex $v \in M$ the isotropy subgroup $\I(v)$ acts $(m-1)$-point transitively on the sphere $S(v,d_1)$, i.~e.
$S(v,d_1)$ is $(m-1)$-point homogeneous under the action of $\I(v)$
for some $m \geq 2$, then $M$ is $m$-point homogeneous.
\end{theorem}

\begin{proof}
Let $M_1$, $M_2$ be arbitrary ordered subsets in $M$ consisting of $m$ vertices such that the distance between two arbitrary points in
$M_1$ is equal to the distance between corresponding two points in $M_2$. Due to the homogeneity
of $M$, we may assume that the first vertices in $M_1$ and $M_2$ are both equal
to $v$.

If $M_1\subset S(v,d_1)\cup \{v\}$, then the same is true for $M_2$
and the conditions of theorem imply that there is $f\in \I(v)$ such that
$f(M_1)=M_2$.

Otherwise, $M_1=M'_1\cup M''_1$ and $M_2=M'_2\cup M''_2$, where $M'_i\subset S(v,d_1)\cup \{v\}$
and $M''_i\subset S(v,d_2)\cup \{-v\}$, $i=1,2$. It is clear that the ordered sets $M'_1\cup (-M''_1)$ and $M'_2\cup (-M''_2)$ of vertices in
$M$ are mutually isometric and both are contained in $S(v,d_1)\cup \{v\}$.
Conditions of the theorem imply that there is $f\in \I(v)$ such that $f(M'_1\cup (-M''_1))=M'_2\cup (-M''_2)$.
Applying several times the second assertion of Lemma \ref{inversion} to the subsets
$-M''_1$ and $-M''_2$, we obtain $f(M_1)=M_2$. Hence, $M$ is $m$-point homogeneous.
\end{proof}

\begin{remark} It is clear that the sphere $S(v,d_1)$ could be replaced by the sphere $S(v,d_2)$ in the statement of Theorem \ref{th:3dist_symm}.
Indeed, $S(v,d_1)$ is the vertex figure for the vertex $v$, while $S(v,d_2)$ is the vertex figure for the vertex $-v$.
\end{remark}

\begin{quest}\label{qu:3dist_symm}
What is the classification of polytopes $P$ in $\mathbb{R}^n$, $n\geq 2$, such that its vertex set $M$ is homogeneous,
symmetric with respect to the center of $P$, and the distances between distinct vertices of $P$
constitute a 3-element set?
\end{quest}

\begin{lemma}
\label{distsph}
For any subset of the sphere
$\{v_1,\dots,v_n\}\subset S^{n-1}(0,1)\subset \mathbb{R}^n$, where $n\geq 2$,
lying in no great subsphere of the sphere $S^{n-1}(0,1)$ of less dimension,
every point $w\in S^{n-1}(0,1)$ is uniquely determined by its distances to the points $v_1,\dots,v_n$.
\end{lemma}

\begin{proof}
Let us denote the spherical distances from the point $w$ to the points $v_i$, $i=1,\dots, n$, by $\varphi_i$.
We consider the points $v_1,\dots,v_n$ and $w$ as unit vectors in $\mathbb{R}^n$. Due to the conditions of the lemma, the first $n$ vectors are linearly independent.
It is accepted to call the inner products $w_i:=w\cdot v_i= \cos\varphi_i$, $i=1,\dots n$, as the
{\it covariant} components of the vector $w$ in the basis $v=(v_1,\dots,v_n)$. Then
$$
(w_1,\dots,w_n)^t=Gr(v)(w^1,\dots,w^n)^t,
$$
where the symbol $^t$ denotes the transposition of matrices,
$w=\sum_{i=1}^nw^iv_i,$ $Gr(v)=(v_i\cdot v_j)$ is nondegenerate {\it Gram matrix} of the basis $v$,
$$
(w^1,\dots,w^n)^t=Gr(v)^{-1}(w_1,\dots,w_n)^t= Gr(v)^{-1}(\cos\varphi_1,\dots, \cos\varphi_n)^t.
$$
The lemma is proved.
\end{proof}

Let us assume that the sphere $S^{n-1}(O,1)$ of radius $1$ centered at the origin $O\in\mathbb{R}^n$
is the circumscribed sphere for all homogeneous $n$-dimensional polyhedra under consideration, and the spherical distance is used as the distance.

\begin{theorem}\label{th:3max}
Let $n\geq 3$ and $P$ be $n$-dimensional convex $(n-1)$-point homogeneous polytope in $\mathbb{R}^n$ with the vertex set $M$ of cardinality
$k\geq n+1$. Then $P$ is $k$-point homogeneous if and only if for any two isometric non-degenerate ordered $n$-tuples of vertices $M_1$ and
$M_2$ in $M$ with the first points $v$, there exists an isometry $f$ from the isotropy subgroup $\I(v)$ such that $f(M_1)=M_2$. If $P$ is also centrally symmetric
then $M$ in this statement can be replaced by an arbitrary maximal by inclusion subset $N\subset M\cap B(v,\pi/2)$, containing no diametrically opposite vertices.
\end{theorem}

\begin{proof}
 We start with proving of the first statement.
The necessity is evident. Let us prove the sufficiency.

Let $M'$, $M''$ be two isometric ordered subsets in $M$ of cardinality $m$, where $n\leq m\leq k$.
There exists a smallest number $l\geq 2$ such that $M'$ (respectively, $M''$) lies in some vector subspace $V'_l$ (respectively, $V''_l$) of dimension $l$ in
$\mathbb{R}^n$.

Assume at first that $l\leq n-1$. Then there exist an $l$-point ordered subset  $M'_l$ in $M'$ and a subset $M''_l$ in $M''$ with the same numbers as
$M'_l$ such that the radius-vectors of the vertices in every of these subsets constitute a basis in $V'_l$ and $V''_l$ respectively.
Since $P$ is $l$-point homogeneous, then there exists a self-isometry $f$ of the polytope $P$  such that
$f(M'_l)=M''_l$. Hence, due to Lemma \ref{distsph},
applied to the spheres $S^{l-1}(0,1)\subset V'_l$ and
$S^{l-1}(0,1)\subset V''_l$, containing respectively the sets $M'$ and $M''$, we obtain $f(M')=M''$.

Assume now that $l=n$. Then there exist an $n$-point ordered subset $M'_l$ in $M'$ and the subset $M''_l$ in $M''$ with the same numbers as $M'_l$, such that the
radius-vectors of the vertices of every such subset constitute bases in
$\mathbb{R}^n$. By the homogeneity of $P$ we may suppose that the first points of subsets $M'_l$ and $M''_l$ coincide with the vertex $v\in M$.
By assumption,
there exists a self-isometry $f\in \I(v)$ of the polytope $P$ such that  $f(M'_l)=M''_l$.
Then by Lemma \ref{distsph}, applied to the sphere $S^{n-1}(0,1)\subset \mathbb{R}^n$, we get $f(M')=M''$.

Let the polytope $P$ be centrally symmetric and $N\subset M\cap B(v,\pi/2)$, where $v\in M$,
is some maximal by inclusion set containing no pair of diametrically opposite vertices. The triangle inequality imply that
$M\cap U(v,\pi/2)\subset N$, and $M\cap U(v,\pi/2)= N$ if and only if
$M\cap S(v,\pi/2)=\emptyset$.

Let $M_1$ and $M_2$ be two isometric nondegenerate ordered $n$-tuples of vertices in $M$. It is clear that every of sets $M_1$ and $M_2$ contains no pairs
of diametrically opposite vertices. The vertices with the same number in  $M_1$ and $M_2$ simultaneously belong or not belong to $M\cap B(v,\pi/2)$
and we will change every pair of vertices of the same name, which do not belong
to this set, by the diametrically opposite pair.
As a result, we obtain (possibly, new) isometric and nondegenerate $n$-tuples of vertices $M'_1$ and $M'_2$ of the polytope $P$, containing in $M\cap B(v,\pi/2)$.

If $M\cap S(v,\pi/2)=\emptyset$, then we set $M''_1=M'_1$ and $M''_2=M'_2$.

If not, then changing, if necessary, some vertices of the same name at $M'_1$ and
$M'_2$ into diametrically opposite ones, we obtain isometric and non-degenerate  $n$-tuples of vertices $M''_1$ and $M''_2$, so that  $M''_1\subset N$.

Due to the assumption of the theorem, there exists a self-isometry $f\in \I(v)$ of the polytope $P$ such that $f(M_1)=M_2$. Moreover, due to Lemma \ref{inversion},
the last equality is equivalent to $f(M''_1)=M''_2.$

The above arguments give the proof of the second part of the theorem,
regardless of whether $M''_2$ is contained in $N$ or not; in the latter case, $M''_2$ can be ignored, leaving only $M''_1\subset N$.
\end{proof}

\begin{corollary}\label{co:if-homog}
Every $n$-dimensional convex $n$-point homogeneous polytope in $\mathbb{R}^n$ which has the vertex set $M$ with the cardinality $k\geq n+1$, is $k$-point homogeneous.
\end{corollary}

\section{Examples}\label{sec.3}

Let us start with a very simple example.

\begin{example}\label{ex:reg_n_sim}
Any regular $n$-dimensional simplex $S$ is $(n+1)$-point homogeneous.
Indeed, the distances between any vertices $S$ coincide and the symmetry group of  $S$ is the group~$S_{n+1}$ of all permutations of the vertex set.
\end{example}

\begin{example}\label{ex:reg_polygon.1}
Any regular $m$-gon $P$ in $\mathbb{R}^2$, $m \geq 3$,  is $m$-point homogeneous.
Indeed, it is clear that $P$ is $2$-point homogeneous. Now the statement follows from Corollary \ref{co:if-homog}.
\end{example}

\begin{example}\label{ex:3-cube}
The vertex set $M$ of the $3$-cube $P$ is (at least) $4$-point homogeneous.
There are three possible  distances between distinct vertices of $P$, $0<d_1<d_2 <d_3$, and $P$ has the center of symmetry.
Moreover, for any vertex $v$ we see that  $S(v,d_1)$ is the vertex set of a regular triangle.
The isotropy subgroup $\I(v)$ acts $3$-point transitively on $S(v,d_1)$. Therefore, $M$ is $4$-point homogeneous by Theorem \ref{th:3dist_symm}.
\end{example}

\begin{example}\label{ex:icos}
The vertex set $M$ of the icosahedron $P$ is (at least) $6$-point homogeneous.
There are three possible  distances between distinct vertices of $P$, $0<d_1<d_2 <d_3$, and $P$ has the center of symmetry.
Next, for any vertex $v$ we see that  $S(v,d_1)$ is the vertex set of a regular pentagon. The isotropy subgroup $\I(v)$ contains all
isometries of this pentagon, hence $\I(v)$  acts $5$-point transitively on $S(v,d_1)$ (see Example \ref{ex:reg_polygon.1}).
Therefore, $M$ is $6$-point homogeneous
by Theorem \ref{th:3dist_symm}.
\end{example}

Corollary \ref{co:if-homog} together with  Examples \ref{ex:3-cube} and \ref{ex:icos} imply

\begin{corollary}\label{co: homsimplepol}
The $3$-dimensional cube is $8$-point homogeneous.
The icosahedron is $12$-point homogeneous.
\end{corollary}

\begin{corollary}\label{co:ortho-homog}
The $n$-orthoplex in $\mathbb{R}^n$ is $2n$-point homogeneous for every $n\in \mathbb{N}$.
\end{corollary}

\begin{proof}
It follows from Theorem \ref{th:3dist_symm} and Corollary \ref{co:if-homog} applied to the $n$-orthoplex.
\end{proof}

\begin{example}\label{ex:2simplex}
Let us check that the convex hull of the union of regular $n$-dimensional simplex and the simplex which is centrally symmetric to it, is $2(n+1)$-point
homogeneous.

As the first simplex $S_1$, we can take the convex hull of $n+1$ points
$v_i\in \mathbb{R}^{n+1}$, $i=1,\dots, n+1$, where $j$-th coordinate of  $v_i$ is $v_{i,j}=-1$ if $j\neq i$ and $v_{i,i}=n$.

One can easily see that the convex hull $P^n$ of $S_1\cup (-S_1)$ is the image of the orthogonal projection (along $(1,\dots,1)\in\mathbb{R}^{n+1}$)
of the standard $(n+1)$-orthoplex $\operatorname{Ort}$ (in $\mathbb{R}^{n+1}$), multipled by $n+1$, onto $n$-dimensional hyperplane
$$
H^{n}=\{x=(x_1,\dots,x_{n+1})\in \mathbb{R}^{n+1}: x_1+\dots x_{n+1}=0\}.
$$

It is easy to see that $P^n$ is homogeneous and satisfies conditions of
Theorem \ref{th:3dist_symm} with $m-1=n$. Then by Theorem \ref{th:3dist_symm} and
Corollary \ref{co:if-homog}, $P^n$ is $2(n+1)$-point
homogeneous.
\end{example}

\begin{remark}
$P^2$ is the regular hexagon and $P^3$ is the cube.
\end{remark}

\begin{example}\label{ex:Gos6.2}
Let us check that the vertex set $M$ of the polytope  $\operatorname {Goss\,}_6$ (see Example~\ref{ex:Gos6.1})
is two-point homogeneous.

At first, we note that the spheres $S(A_1,r)$ are non-empty exactly for $r=1$ and $r=\sqrt{2}$.
Indeed, $d(A_1, A_i)=1$ for $2\leq i \leq 17$ and $d(A_1, A_i)=\sqrt{2}$ for $18\leq  i \leq 27$.
The points $A_2 - A_{17}$ are vertices of a five-dimensional demihypercube (the corresponding hypercube has $32=2^5$ vertices of the form
$(\pm a, \pm a,\pm a,\pm a,\pm a,b)$),
and the points $A_{18} - A_{27}$ are the vertices of the five-dimensional hyperoctahedron ($5$-orthoplex),
which is a facet of the polytope $\operatorname{Goss\,}_6$ (lying in the hyperplane $x_6=-2b$).

It is easy to check that the isotropy subgroup $I(A_1)$ of the group $\Isom(M)=\Symm(\operatorname {Goss\,}_6)$
contains the orthogonal operators of the following two forms:

1) $B(x_1,x_2,x_3,x_4,x_5, x_6)= (x_{\sigma(1)}, x_{\sigma(2)},x_{\sigma(3)}, x_{\sigma(4)},x_{\sigma(5)}, x_6)$, where $\sigma \in S_5$
is any permutation of $(1,2,3,4,5)$;

2) $B(x_1,x_2,x_3,x_4,x_5, x_6)= ( \pm x_1, \pm x_2, \pm x_3, \pm x_4, \pm x_5, x_6)$, where the number of signs ``$-$'' is even.

Now, it is easy to see that a composition of some suitable maps of the above forms moves any given point in
the sphere $S(A_1,1)$ or in the sphere $S(A_1,\sqrt{2})$ to any other given point in the same sphere.
By Proposition \ref{pr:two-point homogeneous}, $M$ is two-point homogeneous.
\end{example}
\smallskip

\begin{example}\label{ex:Gos7.1}
Let us consider a description of the Gosset polytope $\operatorname{Goss\,}_7$.
Let us set it with the coordinates of the vertices in $\mathbb{R}^7$, as it is done in~\cite{Elte12}.
In what follows,
$\overline{a_1,a_2,\dots,a_m}$ means the set of all permutations of the elements $a_1,a_2,\dots,a_m$.

Let us put $a=\frac{\sqrt{2}}{4}$, $b=\frac{\sqrt{6}}{12}$, and $c=\frac{\sqrt{3}}{6}$.
We define the points $B_i \in \mathbb{R}^7$, $i=1,\dots,56$, as follows:
\begin{eqnarray*}
B_1=(0,0,0,0,0,0,3c), \quad\quad B_2=(0,0,0,0,0,4b,c), \quad\quad B_3=(a,a,a,a,a,b,c),\\
B_4 - B_{13}=(\overline{-a,-a,a,a,a},b,c), \quad \quad\quad \quad B_{14} - B_{18}=(\overline{-a,-a,-a,-a,a},b,c),\\
B_{19}-B_{23}=(\overline{2a,0,0,0,0},-2b,c), \quad  \quad\quad \quad \,\, B_{24}-B_{28}=(\overline{-2a,0,0,0,0},-2b,c),\\
B_{29}-B_{33}=(\overline{-2a,0,0,0,0},2b,-c), \quad  \quad\quad \quad \,\, B_{34}-B_{38}=(\overline{2a,0,0,0,0},2b,-c),\\
B_{39} - B_{48}=(\overline{a,a,-a,-a,-a},-b,-c), \,\quad \quad B_{49} - B_{53}=(\overline{a,a,a,a,-a},-b,-c),\\
B_{54}=(-a,-a,-a,-a,-a,-b,-c),\quad \quad\quad \quad \quad\quad\,\,\, B_{55}=(0,0,0,0,0,-4b,-c),  \\
B_{56}=(0,0,0,0,0,0,-3c).\hspace{95.5mm}
\end{eqnarray*}

The Gosset polytope $\operatorname{Goss\,}_7$ is the convex hull of these points.
It is clear that the last 28 points could be obtained from the first 28 point via the central symmetry with respect to the origin $O$.
Hence, the circumradius of this polytope is $3c=\frac{\sqrt{3}}{2}$.
It is clear that the convex hull of the points $B_i$ for $2\leq i \leq 28$ is the Gosset polytope $\operatorname{Goss\,}_6$
(the ``first''  $\operatorname{Goss\,}_6$, that is a vertex figure for the vertex
$B_1\in \operatorname{Goss\,}_7$).
The convex hull of the points $B_i$ for $29\leq i \leq 55$ is a polytope that is symmetric to the previous one (the ``second'' $\operatorname{Goss\,}_6$).

It is easy to check that $d(B_1, B_i)=1$ for $2\leq i \leq 28$, $d(B_1, B_i)=\sqrt{2}$ for $29\leq  i \leq 55$, and $d(B_1, B_{56})=\sqrt{3}$.

It is easy to see that the subgroup of the isometries of $\operatorname{Goss\,}_7$ that preserve all points of the straight line $Ox_7$,
acts transitively on the vertices
$B_i$ for $2\leq i \leq  28$, as well as for $29\leq i \leq  55$.
Indeed, any isometry of $\operatorname{Goss\,}_6$ (that is identified with  the convex hull of the points $B_i$ for $2\leq i \leq 28$) could be
uniquely extended to the isometry of $\operatorname{Goss\,}_7$, that fixes $B_1$ (as well as $B_{56}$).
In particular, the isotropy group $\I(B_1)$
acts transitively on all (non-empty) spheres $S(B_1, r)$. It is clear that there are exactly four non-empty spheres $S(B_1, r)$:
for $r=0$ and $r=6c=\sqrt{3}$ we get one-point sets, whereas for $r= 1$ and $r= \sqrt{2}$ we have two copies of $\operatorname{Goss\,}_6$.
By Proposition \ref{pr:two-point homogeneous}, $\operatorname{Goss\,}_7$ is $2$-point homogeneous.

Moreover, $\operatorname{Goss\,}_7$ is $3$-point homogeneous by Theorem \ref{th:3dist_symm} and Example \ref{ex:Gos6.2}.
\end{example}

\section{Some results on the point homogeneity degree}\label{se_m-point}

\begin{definition}\label{de_degree_hom}
For a given $n$-dimensional convex polytope $P$, let $q$ be a natural number such that $P$ is $q$-point homogeneous,
but is not $(q+1)$-point homogeneous. Then $q$ is called the point homogeneity degree of the polytope $P$.
If $P$ is $m$-homogeneous for all natural $m$ then we define its point homogeneity degree as $\infty$.
\end{definition}

It should be noted that there are  polytopes with the point homogeneity degree $\infty$ (compare with Corollary \ref{co:if-homog}).
For instance, this property is inherent in the following polytopes:
any regular $m$-gon $P$ in $\mathbb{R}^2$ for $m \geq 3$ (Example \ref{ex:reg_polygon.1}), the $n$-dimensional regular simplex (Example \ref{ex:reg_n_sim})
and the $n$-orthoplex for all $n\geq 1$
(Corollary \ref{co:ortho-homog});
the $n$-cube for $n=1,2,3$, the icosahedron (Corollary \ref{co: homsimplepol}).

If the vertex set of a polytope $P$ is homogeneous then the point homogeneity degree of $P$ is $\geq 1$. Otherwise, it is equal to $0$.
We will show that the $n$-cube  has the point homogeneity degree $3$ for $n \geq 4$ (see Corollary \ref{co:homdegcube} below),
the dodecahedron has the point homogeneity degree $2$ (Theorem \ref{th:dod} below). Note also that the cuboctahedron has the point homogeneity degree~$2$,
while all other Archimedean polyhedra have the point homogeneity degree $1$
(Theorems \ref{th:reg_pol} and \ref{th:dod}).

On the other hand, we do not know what is the point homogeneity degree for the $24$-cell, $120$-cell, $600$-cell, as well as for Gosset polytopes
(with the exception of the four-dimensional truncated simplex, whose point homogeneity degree is $2$ due to Proposition~\ref{pr:3-point n-simplex}).

\begin{remark}
It should be noted that the $24$-cell, as well as the $120$-cell and $600$-cell,
contains a $4$-cube (some $16$ vertices of any of these polytope are the vertices of a $4$-cube).
We know that the $4$-cube is not $4$-point homogeneous. It is possible that the $24$-cell, $120$-cell, and $600$-cell have the same property.
\end{remark}

\begin{quest}\label{qu_1}
What is the point homogeneity degree of a given regular polytope $P$?
\end{quest}

\begin{quest}\label{qu_2}
What is the point homogeneity degree of a given semiregular polytope $P$?
\end{quest}

We know that the $6$-dimensional Gosset polytope $\operatorname {Goss\,}_6$ has the point homogeneity degree $\geq 2$ (Example \ref{ex:Gos6.2})
and the $7$-dimensional Gosset polytope $\operatorname {Goss\,}_7$ has the point homogeneity degree $\geq 3$ (Example \ref{ex:Gos7.1}).

\medskip

We are going to find the point homogeneity degree of {\it the $n$-cube for $n \geq 4$.}

\begin{prop}\label{pr:3-point n-cube}
The vertex set of the $n$-cube, $n\geq 2$, is $3$-point homogeneous.
\end{prop}

\begin{proof}
For $n=2$ and $n=3$ it follows respectively from Example \ref{ex:reg_polygon.1} and Corollary~\ref{co: homsimplepol}.
Hence, we consider the standard $n$-cube $C_n$, $n\geq 4$.
The vertices of $C_n$ has the form $(\pm 1, \pm 1, \dots, \pm 1) \in \mathbb{R}^n$.

For any vertex $A$ of $C_n$, we consider the set $N(A)$ of indices $i$, $1\leq i \leq n$, such that $i$-th coordinates of $A$ is $-1$.

Let us consider three vertices $A_0, A_1, A_2$ of $C_n$.
Without loss of generality we may suppose that $A_0=(1,1,1,\dots,1)$, but $A_1$ and  $A_2$ are chosen arbitrarily (and all points are pairwise distinct).
The distance between $A_0$ and $A_1$ ($A_2$) is determined by the cardinality of the set $N(A_1)$ (respectively, $N(A_2)$), but the distance between
$A_1$ and $A_2$
are determined by the cardinality of the set $N(A_1) \cap N(A_2)$.

If we have some another triple of vertices $A'_0$, $A'_1$, $A'_2$ with the same pairwise distances, then there is an isometry of $C_n$
that maps the first triple of the vertices to the second one.
Indeed, since $C_n$ is homogeneous, we may assume that $A'_0=A_0$.

Now, the cardinality of the sets $N(A_1)$, $N(A_2)$, and $N(A_1) \cap N(A_2)$
are the same as the cardinality of the sets $N(A'_1)$, $N(A'_2)$, and $N(A'_1) \cap N(A'_2)$ respectively.
Therefore, there is a substitution of coordinates, that maps the sets $N(A_1)$, $N(A_2)$, and $N(A_1) \cap N(A_2)$
onto the sets $N(A'_1)$, $N(A'_2)$, and $N(A'_1) \cap N(A'_2)$ respectively.
The point $A’_0=A_0$ is fixed under any permutation of coordinates, hence the above permutation generates a desired isometry of $C_n$.
\end{proof}

\begin{prop}\label{pr:4-point n-cube}
The vertex set of the $n$-cube, $n\geq 4$, is not $4$-point homogeneous.
\end{prop}

\begin{proof}
At first, we consider the case $n=4$.
Let us consider the following two $4$-tuples of vertices of the standard $4$-cube:

$$
\Bigl((1,1,1,1),\, (-1,-1,1,1), \,(-1,1,-1,1), \,(-1,1,1,-1)\Bigr)\,
$$
$$
\Bigl((1,1,1,1),\, (-1,-1,1,1), \,(-1,1,-1,1),\, (1,-1,-1,1)\Bigr).
$$
Let us suppose that there is an isometry $f$ of the standard $4$-cube, that maps the first $4$-tuple onto the second one.

It is easy to see that the distances between the corresponding pairs of the points in these $4$-tuples coincide.
We see also that the first three vertices in these $4$-tuples coincide.
Now, any isometry that preserves these first three vertices (in particular, $f$)
should be the identical map that is impossible.

If $n>4$, we can consider two $4$-tuple of vertices of the standard $n$-cube, such that the first four coordinates are the same
as the above $4$-tuples for the standard $4$-cube have,
but all other coordinates are equal to $1$.
The same argument shows that the standard $n$-cube, $n > 4$, is not $4$-homogeneous.
\end{proof}
\medskip

We see that the vertex set of the $n$-cube is $3$-point homogeneous but is not $4$-point homogeneous for all $n \geq 4$, hence, we get

\begin{corollary}\label{co:homdegcube}
The point homogeneity degree of the $n$-cube is $3$ for  all $n \geq 4$.
\end{corollary}

\begin{remark}
It should be noted that the vectors in both $4$-tuples in the above example generate the following Hadamard matrices:
$$
\left(
\begin{array}{rrrr}
  1 & 1 & 1 & 1 \\
  -1 & -1 & 1 & 1 \\
  -1 & 1 & -1 & 1 \\
  -1 & 1 & 1 & -1 \\
\end{array}%
\right)
\quad \mbox{ and } \quad
\left(
\begin{array}{rrrr}
  1 & 1 & 1 & 1 \\
  -1 & -1 & 1 & 1 \\
  -1 & 1 & -1 & 1 \\
  1 & -1 & -1 & 1 \\
\end{array}%
\right).
$$
The necessary information about Hadamard matrices and related results can be found, for example, in \cite{HeWa}.
\end{remark}

\smallskip

We are going to find the point homogeneity degree of {\it the $n$-dimensional demihypercube}.
We consider the standard $n$-dimensional demihypercube $DC_n$, $n\geq 1$, as the convex hull
of the points $(\pm 1, \pm 1, \dots, \pm 1) \in \mathbb{R}^n$, where the quantity of the signs ``$-$'' is even
(if we consider the odd quantity of such signs, then we get another one $n$-dimensional demihypercube).
It is clear that $DC_n$ has the center of symmetry if and only if $n$ is even.
The symmetry group of $DC_n$ consists of those symmetries of the standard $n$-cube that preserved $DC_n$
(it contains all permutations of coordinates and all inversions of signs for even quantity of coordinates).

It is easy to see that $DC_2$ (that is a segment) is two-point homogeneous and $DC_3$ (that is a regular simplex) is $4$-point homogeneous.
Therefore, the point homogeneity degree of $DC_n$ is $\infty$ for $n =1,2,3$.

\begin{prop}\label{pr:3-point d-n-cube}
The vertex set of $DC_n$, $n\geq 4$, is $3$-point homogeneous.
\end{prop}

\begin{proof}
The vertices of $DC_n$ has the form $(\pm 1, \pm 1, \dots, \pm 1) \in \mathbb{R}^n$ with even quantity of the signs ``$-$''.

For any vertex $A$ of $DC_n$, we consider the set $N(A)$ of indices $i$, $1\leq i \leq n$, such that $i$-th coordinates of $A$ is $-1$.

Let us consider three vertices $A_0, A_1, A_2$ of $DC_n$.
Without loss of generality we may suppose that $A_0=(1,1,1,\dots,1)$, but $A_1$ and  $A_2$ are chosen arbitrarily (and all points are pairwise distinct).
The distance between $A_0$ and $A_1$ ($A_2$) is determined by the cardinality of the set $N(A_1)$ (respectively, $N(A_2)$), but the distance between
$A_1$ and $A_2$
are determined by the cardinality of the set $N(A_1) \cap N(A_2)$.

If we have some other triple of vertices $A'_0$, $A'_1$, $A'_2$ with the same pairwise distances, then there is an isometry of $DC_n$
that maps the first triple of the vertices to the second one.
Indeed, since $DC_n$ is homogeneous, we may assume that $A'_0=A_0$.

Now, the cardinality of the sets $N(A_1)$, $N(A_2)$, and $N(A_1) \cap N(A_2)$
are the same as the cardinality of the sets $N(A'_1)$, $N(A'_2)$, and $N(A'_1) \cap N(A'_2)$ respectively.
Therefore, there is a substitution of coordinates, that maps the sets $N(A_1)$, $N(A_2)$, and $N(A_1) \cap N(A_2)$
onto the sets $N(A'_1)$, $N(A'_2)$, and $N(A'_1) \cap N(A'_2)$ respectively.
The point $A’_0=A_0$ is fixed under any permutation of coordinates, hence the above permutation generates a desired isometry of $DC_n$.
\end{proof}

\begin{prop}\label{pr:4-point d-n-cube}
The vertex set of $DC_n$, $n\geq 4$, is not $4$-point homogeneous.
\end{prop}

\begin{proof}
Let us consider the following two $4$-tuples of vertices of $DC_n$:

$$
\Bigl((1,1,1,1, \dots),\, (-1,-1,1,1, \dots), \,(-1,1,-1,1, \dots), \,(-1,1,1,-1, \dots)\Bigr)\,
$$
$$
\Bigl((1,1,1,1, \dots),\, (-1,-1,1,1, \dots), \,(-1,1,-1,1, \dots),\, (1,-1,-1,1, \dots)\Bigr),
$$
where $\dots$ means $n-4$ elements $1$.
Let us suppose that there is an isometry $f$ of $DC_n$, that maps the first $4$-tuple onto the second one.

It is easy to see that the distances between the corresponding pairs of the points in these $4$-tuples coincide.
We see also that the first three vertices in these $4$-tuples coincide.
Now, the isometry that preserves the first two vertices (in particular, $f$)
could only interchange the first two coordinates as well as the last two coordinates of all points.
On the other hand, from the third vertices in the above $4$-tuples, we see that neither the first two coordinates nor the last two coordinates
could be interchanged by~$f$.
Hence, $f$ should preserve the first four coordinates of any vertex that is impossible and $DC_n$ is not $4$-homogeneous  $n \geq 4$.
\end{proof}
\medskip

We see that the $n$-dimensional demihypercube $DC_n$ is $3$-point homogeneous but is not $4$-point homogeneous for all $n \geq 4$.
Therefore, we get

\begin{corollary}\label{co:homdegdemicube}
The point homogeneity degree of the $n$-dimensional demihypercube is $3$ for  all $n \geq 4$.
\end{corollary}

We shall find the point homogeneity degree of {\it the truncated $n$-dimensional simplex}.

The standard truncated $n$-dimensional simplex $TS_n$ can be represented as the convex hull of points
$B_{i,j} \in \mathbb{R}^{n + 1}$, $i,j=1,\dots,n+1$, $i<j$, such that
all coordinates of the point $B_{i,j}$ are zero except for its $i$-th and $j$-th coordinates, which are equal to~$1$.
It is clear that there are exactly $C_n^2=n(n-1)/2$ such points.
The distance between (different) points $B_{i,j}$ and $B_{s,t}$ is $\sqrt{2}$ for $i=s$ or $j=t$ and $2$ for $\{i,j\}\cap \{s,t\}=\emptyset$.
The isometry group of the corresponding simplex (the convex hull of points $A_i \in \mathbb{R}^{n + 1}$, $i = 1,\dots,n + 1$, such that
all coordinates of $ A_i $ are zero except for its $i$-th coordinate, which is $2$)
acts transitively on the vertex set of the truncated simplex under consideration.
Let $M$ be the vertex set of  $TS_n$. The above arguments show that $M$ is a finite homogeneous two-distance set.
Moreover, $TS_n$ is uniform in every dimension \cite{Cox34}.

Note that $TS_1$ is a one-point set, $TS_2$ is a regular triangle, $TS_3$ is an orthoplex (octahedron), $TS_4$ is a semiregular polytope.

The isometry group of $TS_n$ is the symmetric group $S_{n+1}$ for $n=2$ and $n \geq 4$ (the group of permutations of the symbols $1,2,\dots, n, n+1$),
see Theorem 1 in \cite{BMOW2018}. The isometry group of $TS_3$ is the isometry group of the octahedron $\mathbb{Z}_2^3 \rtimes S_3=\mathbb{Z}_2 \rtimes S_4$.

We will denote the vertex $B_{i,j}$ also by the symbol $\{i,j\}$ (the set of two indices).
Let us fix the vertex $B_{1,2}=\{1,2\}$. We see that the sphere $S(\{1,2\},\sqrt{2})$ (respectively, the sphere $S(\{1,2\},2)$)
consists of vertices  $\{i,j\}$ such that the set
$\{1,2\}\cap \{i,j\}$ contains exactly one element (respectively $\{1,2\}\cap \{i,j\}=\emptyset$).
The isotropy group $\I(\{1,2\})$ is the direct product $S_2 \times S_{n-1}$ (independent permutations of the elements $1,2$ as well as the elements
$3,4,\dots,n+1$).
Obviously, $\I(\{1,2\})$ acts transitively on spheres $S(\{1,2\},\sqrt{2})$ and $S(\{1,2\},2)$). Therefore $M$ is $2$-point homogeneous.

Let us show that $M$ is not $3$-point homogeneous. For this we consider the following triples of vertices:
$$
\bigl(\{1,2\}, \{1,3\},  \{2,3\}\bigr), \qquad \bigl(\{1,2\}, \{1,3\},\{1,4\}\bigr).
$$

If the isometry $f$ maps the first triple onto the second one, then from $f(\{1,2\})=\{1,2\}$ we see that $f$ either
interchanges the coordinates with numbers $1$ and $2$ or preserves both these coordinates. Next, from $f(\{1,3\})=\{1,3\}$ we see that $f$ preserves the
coordinates with numbers $1$, $2$ and $3$.
Therefore, the equality $f(\{2,3\})=\{1,4\}$ is impossible.
In this construction, we used four distinct coordinates of the vertices, that is important.

Hence, we have proved the following result.

\begin{prop}\label{pr:3-point n-simplex}
The point homogeneity degree of the truncated $n$-dimensional simplex is $2$ for $n\geq 4$ and is $\infty$ for $n=1,2,3$.
\end{prop}

\section{On $2$-point homogeneous polyhedra with equal edge lengths}\label{se_m-point.3-dim}

Recall that any regular polytope in $\mathbb{R}^3$ is $2$-point homogeneous
(see Proposition \ref{pr:3dim.2point.reg}).
The main results of this section are the proofs of Theorems  \ref{th:reg_pol} and \ref{th:edge}.

In the arguments below, we will use the spherical distance between the points.
Every polyhedron is assumed to be inscribed into a sphere with center in the origin $O$ and unit radius.

\begin{theorem}
\label{th:dod}
The point homogeneity degree of the dodecahedron $D$ is $2$ {\rm(}i.~e. $M$ is $2$-point homogeneous but not $3$-point homogeneous{\rm)}.
\end{theorem}

\begin{proof}
Let us recall some property of the dodecahedron $D$. We can assume that the vertices of $D$ are the centers of the faces of an icosahedron $I$
with the center $O$, so that the inscribed sphere for $I$ is the circumscribed sphere for $D$,
while the spherical distance $d_1$ between adjacent vertices for $D$ is the angle between outer normals of adjacent faces for $I$.

The dihedral angle under the edge of the icosahedron is equal to  $\alpha=2\arcsin(\varphi/\sqrt{3})$ (see Table \ref{table0} or \cite{BerNik21}). Therefore
$$
d_1=2\pi-(2\pi/2+\alpha)=\pi-\alpha=2(\pi/2-\arcsin(\varphi/\sqrt{3}))=2\arccos(\varphi/\sqrt{3}),
$$
$$
\cos d_1=\frac{2\varphi^2}{3}-1= \frac{2\varphi^2-3}{3}= \frac{3+\sqrt{5}-3}{3}=\frac{\sqrt{5}}{3}.
$$
By the spherical cosine theorem, for the next spherical distance
$$
\cos d_2= \cos^2d_1+ \cos\left(\frac{2\pi}{3}\right)\sin^2d_1=
\cos^2d_1 - \frac{1}{2}(1-\cos^2d_1)=\frac{3\cos^2d_1-1}{2}=\frac{1}{3}.
$$
All possible spherical distances between vertices of the dodecahedron are equal to
$$
0< \arccos\left(\frac{\sqrt{5}}{3}\right) < \arccos\left(\frac{1}{3}\right)< \arccos\left(\frac{-1}{3}\right)< \arccos\left(\frac{-\sqrt{5}}{3}\right)<d_5=\pi.
$$

The isotropy subgroup $\I(v)$ of the vertex $v\in M \subset D$ in the symmetry group of $D$ consists of $6$ isometries,
three rotations around the axis $v(-v)$ by angles $0$, $2\pi/3$, $4\pi/3$ and three mirror reflections relative to every plane,
passing through $O$ and one of the edges of the dodecahedron, which includes $v$.
As a corollary, the group $\Isom(D)$ of all isometries of the dodecahedron has order $6\cdot 20=120$.

The spheres $S(v,d_1)$, $S(v,d_2)$, $S(v,d_3)$, $S(v,d_4)$ in $S^2(O,1)$ contain respectively $3$, $6$, $6$, $3$ vertices of the dodecahedron $D$,
while $D$ has also the vertices $v$ and $-v$. The isotropy group $\I(v)$ acts twice transitively on sets $M\cap S(v,d_1)$ and
$M\cap S(v,d_4)$ and simply transitively on sets
$M\cap S(v,d_2)$ and $M\cap S(v,d_3)$.

We get that the vertex set $M$ of $D$ is $2$-point homogeneous
from the above arguments and the homogeneity of the dodecahedron $D$.

The set $M\cap S(v_1,d_2)$, where $v_1\in M$, has vertices $v_2$,
$v_3$, $v'_3$ with the spherical distances between any two distinct vertices in $\{v_2,v_3,v'_3\}$ equal to $d_3$. Consequently,
the ordered triples $M_1=\{v_1,v_2,v_3\}$ and $M'_1=\{v_1,v_2,v'_3\}$
are isometric. At the same time, there is no isometry of $M$ and $S^2(0,1)$
such that $f(M_1)=M'_1$ because the isotropy subgroup $\I(v_1)$ acts simply transitively on $M\cap S(v_1,d_2)$ due to the above observations
(see also Fig. \ref{Fig_arch} a)).
\end{proof}

\smallskip

\begin{proof}[Proof of Theorem \ref{th:reg_pol}]
We know that
the tetrahedron, cube, octahedron, icosahedron  are at least $3$-point homogeneous polyhedra
(see Example \ref{ex:reg_n_sim}, Corollary \ref{co:ortho-homog} and Corollary \ref{co: homsimplepol}).
Hence, it suffices to apply Corollary \ref{co:if-homog} to these four regular polyhedra.
The corresponding result for
the dodecahedron follows from Theorem~\ref{th:dod}.
\end{proof}

\begin{figure}[t]
\vspace{10mm}
\begin{minipage}[h]{0.32\textwidth}
\center{\includegraphics[width=0.79\textwidth]{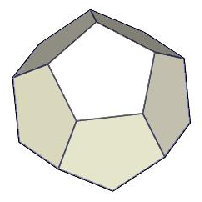} \\ a)}
\end{minipage}
\begin{minipage}[h]{0.32\textwidth}
\center{\includegraphics[width=0.74\textwidth]{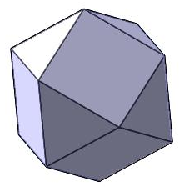} \\ b)}
\end{minipage}
\begin{minipage}[h]{0.33\textwidth}
\center{\includegraphics[width=0.79\textwidth]{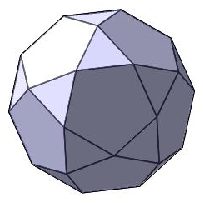} \\ c)}
\end{minipage}
\caption{
a) dodecahedron;
b) cuboctahedron;
c) icosidodecahedron.
}
\label{Fig_arch}
\end{figure}

\medskip

Now, we are going to study semiregular polyhedra.
Recall that all semiregular polytopes are ($1$-point) homogeneous.

\begin{lemma}\label{le:semrer3}
If a semiregular polyhedron $P \in \mathbb{R}^3$ is $2$-point homogeneous and is not regular, then $P$ is possibly the cuboctahedron or the icosidodecahedron.
\end{lemma}

\begin{proof}
All semiregular (right) prisms different from the cube are not $2$-point homogeneous because there is no isometry of the polyhedron,
moving the horizontal edge to the vertical one. Analogously, all semiregular antiprisms different from the octahedron are not $2$-point homogeneous.

Among $13$ Archimedean polyhedra, possibly only two, the cuboctahedron and the icosidodecahedron (see Fig. \ref{Fig_arch}),
are two-point homogeneous because all the other have different
edges incident to nonisometric pairs of faces containing them.
\end{proof}
\smallskip

The isotropy subgroups of all vertices of the cuboctahedron and the icosahedron coincide as subgroups of Euclidean isometry group:
they are generated by mirror reflections with respect to two mutually perpendicular planes and include also only the identity mapping and
the rotation by the angle $\pi$ around straight line that is the intersection of these planes.
The cuboctahedron and icosidodecahedron are, respectively, the convex hulls of the set of midpoints of all edges of the cube (or octahedron)
and of the dodecahedron (or icosahedron) and therefore have $12$ and $30$ vertices, respectively~\cite{BerNik21}.
Therefore, their isometry groups have orders $48$ and $120$ which are equal to the orders of the isometry groups of the cube
(or octahedron) and of the dodecahedron (or icosahedron).

\begin{prop}\label{pr:co}
The {\rm(}semiregular{\rm)} cuboctahedron is $12$-point homogeneous {\rm(}hence, its point homogeneity degree is $\infty${\rm)}.
\end{prop}

\begin{proof}
The cuboctahedron, being centrally symmetric, has $4$ possible nonzero spherical distances between the vertices: $\frac{\pi}{3} < \frac{\pi}{2}< \frac{2\pi}{3}< \pi$.
For every vertex $v\in M$, the sets $M\cap S(v,\frac{\pi}{3})$,
$M\cap S(v,\frac{\pi}{2})$, $M\cap S(v,\frac{2\pi}{3})$ have respectively  $4$, $2$, $4$ vertices; on these sets the isotropy subgroup
$\I(v)$ acts respectively simply transitively, twice transitively, and simply transitively. Consequently the cuboctahedron is two-point homogeneous.

There are two ordered pairs of vertices $(v_1,v_2)$ and $(w_1,w_2)$ in $M\cap S(v,\frac{\pi}{3})$ and one point $w\in M\cap S(v,\frac{\pi}{2})$ such that
$$
d(v_1,v_2)=d(w_1,w_2)=\frac{\pi}{2}; \quad d(v_i,w_i)=\pi/3, \quad i=1,2; \quad d(v_1,w_2)=d(v_2,w_1)=2\pi/3,
$$
so that $v$ is the midpoint of spherical segments $[v_1,w_2]$ and $[v_2,w_1]$.
In addition,
$$
d(w,w_i)=\frac{\pi}{3},\quad d(w,v_i)=\frac{2\pi}{3},\quad i=1,2,
$$
so that $w_i$ is the midpoint of spherical segments  $[w,v_i]$, $i=1,2$. The set of all mentioned $6$ vertices is a maximal by inclusion subset
$N\subset M\cap B(v,\frac{\pi}{2})$, containing no pair of diametrically opposite points.

Next, we list all non-degenerate triangles with vertices in $N$ and the chosen vertex $v$.
There are only two isosceles triangles with base $[v,w]$ and lateral sides $\pi/3$ and two triangles with this base and other sides with lengths
$\pi/3$ and $2\pi/3$. In addition, there are only two equilateral triangles with sides $\pi/3$ and two isosceles triangles with base length
$\pi/2$ and lateral sides $\pi/3$ without the vertex $w$.

It is easy to check that for any two different isometric ordering of the vertices for these triangles with the first vertex $v$,
there exists a self-isometry of the cuboctahedron from $\I(v)$, moving one ordering to another one.
In the first case (with the vertex $w$ included) this is only the mirror reflection with respect to $[v,w]$,
in the second case, either the same or a different mirror symmetry, or the central symmetry with respect to the vertex $v$.

As a consequence of this observation and Theorem \ref{th:3max}, the cuboctahedron is $12$-point homogeneous.
\end{proof}

\begin{prop}\label{pr:ic}
The icosidodecahedron is not $2$-point homogeneous {\rm(}hence, its point homogeneity degree is $1${\rm)}.
\end{prop}

\begin{proof} We give two different proofs. Without loss of generality we can suppose that the origin $O$ in $\mathbb{R}^3$ is the center of the isocidodecahedron $P$.

{\it The first proof.} We can represent the icosidodecahedron $ID$ as a convex hull of the edge midpoints of a suitable dodecahedron $D$ in $\mathbb{R}^3$.
Let us consider some vertex $v$ of $ID$, i.~e. a midpoint of some edge $e$ of $D$. The isotropy subgroup $\I(v)$ of $ID$ is generated by two reflections
in $\mathbb{R}^3$ with respect to planes trough $v$ and the origin,
one of which contains  $e$, and another one is orthogonal to it. In both these planes, there are two diametrically
opposite midpoints of the edges of the dodecahedron at a spherical distance $\pi/2$ from $v$. This together with Proposition \ref{pr:two-point homogeneous}
imply the proposition, since there is no
isometry from $I(v)$, sending one of the planes to the other.

{\it The second proof.} Let us consider two faces of the icosidodecahedron $ID$ with a common edge $l$ (one of them is triangular and the other is pentagonal),
and let $C$ be the midpoint of the edge $l$.
In the triangular face and in the pentagonal face, take the vertices $A$ and $B$ respectively, that are the furthest vertices from the point $C$ in their faces.
It is clear that the distances from $A$ and $B$ to $C$ are different. Therefore, the icosidodecahedron does not have an isometry that interchanges $A$ and $B$.
Indeed, such an isometry must preserve the midpoint of the segment $[A,B]$, and hence also the line passing through this midpoint and the origin.
In particular, they will preserve the triangle $ABC$. It is clear that this is impossible.
The proposition is proved.
\end{proof}

\begin{remark}
Let assume that $v$ is a vertex of the icosidodecahedron, i.~e. is the midpoint of some edge $e$ of the dodecahedron.
Let us consider all $4$ dodecahedron faces intersecting with the edge $e$ and take the midpoint of all edges of these faces, i.~e.
the corresponding vertices of the icosidodecahedron. We shall obtain 15 such vertices plus two more its vertices on the distance
$\frac{\pi}{2}$ from $v$ outside these 4 faces. The corresponding picture of these vertices evidently implies the assertion in the first
proof of Proposition~\ref{pr:ic} that $\I(v)$ is generated by reflections in two mutually orthogonal planes.

It is possible to calculate that the spherical distances from $v$ to another these 14 vertices are equal to
$d_1=\frac{\pi}{5}< d_2=\frac{\pi}{3}< d_3=\frac{\pi}{2}$. So other possible
spherical distances between the vertices of the icosidodecahedron are
$d_4=\frac{2\pi}{3}< d_5=\frac{4\pi}{5}< \pi$.

From the above, we can conclude that there are $6$ geodesics (the great circles) of the circumsphere of the icosidodecahedron $ID$,
passing through the vertices $v$ and $-v$
of $ID$:
$2$ of them
each contain $10$ vertices, $2$ others each contain $6$ vertices, and the remaining $2$ each contain $4$ vertices from $ID$.
\end{remark}

\begin{proof}[Proof of Theorem \ref{th:edge}]
Since every edge of the polyhedron $P$ is the intersection of two its adjacent faces, then it follows from the conditions of theorem
that there exist no more than two face families such that the isometry group $I(P)$ acts transitively only on the faces from one family.

At first suppose that there are two such families, namely $\mathcal{F}_1$ and $\mathcal{F}_2$. Take any face $F\in \mathcal{F}_1$ and vertices
$A,B,C\in F$ on two adjacent edges of $F$. Since $d(A,B)=d(B,C)$ and $P$ is 2-point homogeneous, then there exists an isometry $f\in I(P)$ such that $f(A)=B$, $f(B)=C$.
Since $f$ cannot map a face from $\mathcal{F}_2$ onto a face from $\mathcal{F}_1$, then $f(F)=F$ and the inner angles of the face $F$ by the vertices $A$ and $B$ coincide.
This argument is also applied to any adjacent edges of the face $F$. Consequently the face  $F$ and all faces of the family $\mathcal{F}_1$
are regular pairwise congruent polygons. The same argument and statements are applied to the faces of the family $\mathcal{F}_2$.
Now in the consequence of $2$-point homogeneity (hence homogeneity) of the polytope $P$, it is semiregular.
On the base of the conducted investigation, in particular by Propositions~\ref{pr:co} and \ref{pr:ic}, $P$ is the cuboctahedron.

Now assume that $I(P)$ acts transitively on the family $\mathcal{F}$ of all faces of the polytope $P$.
Then one can easily deduce from condition of theorem that the sum $\sigma$ of inner angles of two intersecting by joint edge faces under the same
vertex of this edge depends nor an edge nor its vertex.

Let $F_1\in \mathcal{F}$ and $AB$, $BC$ be edges of the face $F_1$. By conditions of theorem, there exists $f\in I(P)$ such that $f(A)=B$, $f(B)=A$. If moreover
$f(F_1)=F_1$, then also $f(F_2)=F_2$, where $F_2$ is another face of the polyhedron $P$ with the edge $AB$ and $f$
inverses the orientation of the space  $\mathbb{R}^3$. This and the transitivity of the group $I(P)$ on edges imply that for every face
$F\in \mathcal{F}$ and every its edge $e$ there is $g\in I(P)$ such that $g(F)=F$ and $g$ transposes the vertices of the edge $e$.
It follows from here that all faces of the polytope $P$ are congruent regular polygons and $P$ is regular polytope.

But (for now) a priori it may be that for any $f\in I(P)$ such that $f(A)=B$, $f(B)=A$, it must be $f(F_2)=F_1$ and $f(F_1)=F_2$.
Then $f$ preserves the orientation of the space $\mathbb{R}^3$ and in view of transitivity of the group $I(P)$ on edges,
for any face $F\in \mathcal{F}$ and every its edge $e$ there exists a $g\in I(P)$ such that $g$ transposes the vertices of the edge $e$ and the faces which include $e$.
Then $f$ (respectively,~$g$) is the central symmetry of the sphere $S(0,1)$ relative to the midpoint of the edge $AB$ (respectively,~$e$).
Applying such procedures consecutively
to the edges $AB$ and $BC$, we obtain that inner angles of the face $F_1$ under vertices $A$ and $C$ are equal.
Then, if the number of vertices of the face $F$ is odd, then $F$ is a regular polygon and $P$ is a regular polyhedron.

Let the number $k$ of the face $F$ is even. Then $k\geq 4$ and the sum of inner
angles of the polygon $F$ is equal to
$$
\frac{k\sigma}{2}=(k-2)\pi\Longrightarrow \sigma = \frac{ 2(k-2)\pi}{k}\geq \pi,
$$
$$
k\geq 6\Longrightarrow \sigma = \frac{2(k-2)\pi}{k}\geq \frac{4\pi}{3}.
$$

If the quality of faces of the polytope $P$ with joint vertex $v\in M$ is  $l=4$, then the quality of pairs of adjacent faces (without joint faces)
with common vertex $v$ is equal to $2$ and the sum of inner angles under $v$ of all faces with the vertex $v$ is equal to $2\sigma\geq 2\pi$, what is impossible.
Hence $l=3$ and since the sum of inner angles of two intersecting by joint edge faces under the same vertex $v$ of this edge is equal to $\sigma$, then all
these three inner angles are equal to $\sigma/2$. Consequently all faces and the polyhedron $P$ are regular.

Furthermore $k\geq 6$ is impossible, otherwise in consequence of the second  inequality, written above in the row,
the sum of all inner angles of three polygons under the joint vertex $v$ is equal to
$$
\frac{3\sigma}{2}\geq \frac{3}{2}\cdot \frac{4\pi}{3}=2\pi,
$$
what is impossible. Thus, in the case under consideration it must be $k=4$, $l=3$, $\sigma=\pi$, and $P$ is a (regular) cube.
\end{proof}

\section{On $m$-point homogeneous polyhedra with several edge lengths}\label{ne_m-point.3-dim}

\begin{example}
\label{ex:pr-antpr}
The right three-dimensional (anti)prism with regular bases is three-point homogeneous if the length of the lateral edge is longer than the sides and diagonals of the bases.
\end{example}

\begin{example}\label{ex:simplsev}
Let us consider a simplex $S$ in $\mathbb{R}^3$ with given three edge lengths $0 < a < b < c$, $c^2<a^2+b^2$, such that every pair of edges without
common vertex have equal Euclidean lengths and each of the number $a, b, c$ is the length of edges
for some pair. Hence we have a 3-parameter family of simplices  in~$\mathbb{R}^3$. One can prove that every such simplex is 4-point
homogeneous. Hence we have a 3-parameter family of 4-point homogeneous simplices
(4-point sets) in $\mathbb{R}^3.$
\end{example}

\begin{example}\label{ex:setsev}
If a finite metric space $(M,d)$ is homogeneous and for some its point the distances to all other points are pairwise distinct, then $(M,d)$ is $k$-point
homogeneous for every $k\geq 1$. Indeed, it suffices to move only one point of the first $k$-tuple to the corresponding point of the second $k$-tuple with a suitable
isometry and all other points will automatically be placed in the required places.
\end{example}

\begin{prop}\label{pr:octsev}
Let $P$ be a $3$-dimensional polyhedron, isomorphic to the octahedron, with diametrically opposite pairs of vertices $A,B$; $C,D$; $E,F$; $d_1,d_2$
be positive numbers such that $d_1<d_2$, $d_1+d_2=\pi$. Let us set distances
$$
d(A,B)=d(C,D)=d(E,F)=\pi;
$$
$$
d(A,C)=d(A,E)=d(B,D)=d(B,F)=d(C,E)=d(D,F)=d_1;
$$
while the rest 6 distances to be equal to $d_2$. Then $(P,d)$ is realized as the three-dimensional convex centrally symmetric polyhedron in
$\mathbb{R}^3$ with vertices on $S^2(0,1)$ and with indicated spherical distances between the vertices. Moreover, $(P,d)$ is a $6$-point homogeneous polyhedron.
\end{prop}

\begin{proof}
Let us consider the following matrix $G$ with the value $\alpha = \cos d_1$:
\begin{displaymath}
G=
\left(\begin{array}{ccc}
1 & \alpha & \alpha\\
\alpha & 1 & \alpha \\
\alpha & \alpha & 1
\end{array}\right).
\end{displaymath}
The matrix $G$ is positive definite.
Therefore, there are linearly independent vectors $A, C, E\in S^2(0,1)\subset \mathbb{R}^3$
with the Gram matrix $G$. For example, we can take the vectors whose components are lines of the following matrix
\begin{displaymath}
V=
\frac{1}{\sqrt{1+2\beta^2}}\left(\begin{array}{ccc}
1 & \beta & \beta\\
\beta & 1 & \beta \\
\beta & \beta & 1
\end{array}\right),
\end{displaymath}
where $\beta>0,$ $\beta(2+\beta)/(1+2\beta^2)=\alpha$.
It is clear that the convex hull of the set $M=\{A,C,E,B=-A,D=-C,F=-E\}$ is a realization of the space $(P,d)$.

From the form of the matrix $G$ it follows that every transposition of the ordered pairs $\{A,B\}$, $\{C,D\}$, $\{E,F\}$ and
the central symmetry are isometries of the polyhedron $P$. Consequently it is two-point homogeneous.
For the ordered pair $\{A,C\}$, there are only two non-degenerate triples  $\{A,C,E\}$ and $\{A,C,F\}$, and they are not isometric.
For the ordered pair $\{A,F\}$, there are only two non-degenerate triples $\{A,F,D\}$ and $\{A,F,C\}$
and they are not isometric. Consequently, $(P,d)$ is $3$-point and $6$-point homogeneous by Theorem \ref{th:3max}.
\end{proof}

Let us consider some general idea how to disprove that a subset $M \subset \mathbb{R}^3$ is $3$-point homogeneous.

Let us suppose that $M$ is $3$-point homogeneous, hence, it is $2$-point homogeneous and $1$-point homogeneous.
Without loss of generality we may assume that $M \subset S(O,1)$, where $O$ is the origin in $\mathbb{R}^3$.
Now, let us take points $A,B \in M$, $A\neq B$, such that $O$ is not on the straight line $AB$.
Let $E$ be the midpoint of the segment $[A,B]$ and $L$ be a plane through $E$ and orthogonal to $AB$.

\begin{prop}\label{pr:non-3-point}
In the above assumptions, let us suppose that points $C,D \in M$ are such that $d(A,C)=d(A,D)$, $d(B,C)=d(B,D)$, and $C\neq D$. Then
the symmetry with respect to the plane $OAB$ in $\mathbb{R}^3$ is an isometry of $M$ that interchanges the points $C$ and $D$.
\end{prop}

\begin{proof}
Let us consider the orthogonal projection
$\psi: \mathbb{R}^3 \rightarrow L$ and $M'=\psi(M)$.
It should be noted that $\psi(A)=\psi(B)=E$, $\psi(O)=O$ (since $A,B \in M \subset S(O,1)$), and $O$ is the barycenter of $M'$.
It is clear that $C\neq D$ implies $\psi(C)\neq \psi(D)$ ($\psi(C)=\psi(D)$, $d(A,C)=d(A,D)$, and $d(B,C)=d(B,D)$ imply $C=D$).
Since $M$ is supposed to be $3$-point homogeneous and $(A,B,C)$, $(A,B,D)$ are triples of points from $M$ with corresponding equal distances, then
there is an isometry of $M$ that sends the triple $(A,B,C)$ to the triple $(A,B,D)$.

Applying $\psi$, we see that there is an isometry of $M'$ that preserves $E$ and sends $\psi(C)$ to $\psi(D)$.
Since $E\neq O$ and $O$ is the barycenter of $M'$, then the latter isometry should be the symmetry with respect to the straight line $EO$.
From this we easily get that
$C$ and $D$ are symmetric to each other with respect to the plane $OAB$ (recall that $d(A,C)=d(A,D)$ and $d(B,C)=d(B,D)$).
\end{proof}

\begin{remark}
If $M'$ in the above proof is not symmetric with respect to the straight line $EO$, then the equalities $d(A,C)=d(A,D)$ and $d(B,C)=d(B,D)$
for some points $C,D \in M$ imply that $C=D$.
\end{remark}

\section{Conclusion}\label{concl}

The notion of $n$-point homogeneity for a metric space was introduced by
G.~Birkhoff in \cite{Birk1944} by the following sentence:
``Any isometry between two (sub)sets of $n$ points or fewer points can be extended to a self-isometry of space''.
In the same paragraph he noticed that the Urysohn space \cite{Urys1927} has $n$-point homogeneity for every finite order $n$.

This posthumous paper by P.~Urysohn contains his construction of remarkable complete separable metric space $U$ which besides the
$n$-point homogeneity for every finite order $n$ is universal in the sense that every separable metric space is isometric to a subspace of $U$.
Moreover, any space with these properties is isometric to $U$ \cite{Urys1927}.

Recently were appeared many interesting results on the space $U$.
Let us mention only three of them. Any isometry between two compact subsets of the space $U$ can be extended to an isometry of $U$ onto itself \cite{Bog2002}.
The topological group $I(U)$ of all isometries of $U$ is universal in the sense that it contains an isomorphic copy of every topological group with
a countable base \cite{Usp1990}. The space $U$ is homeomorphic to the Hilbert space~$l_2$~\cite{Usp2004}.

It should be noted that there is no sense to classify (even finite) $m$-point homogeneous metric spaces up to isometry or similarity without additional
restrictions, because any such metric space $(M,d)$ generates a class of $m$-point homogeneous metric spaces of the following type:
$(M,\psi\circ\, d)$, where $\psi$ is any increasing concave (convex up) function such that
$\psi(0)=0$. For instance, one can take $\psi(t)=\ln(1+t)$, $t\geq 0$, or
$\psi(t)=2\sin (t/2)$, $0\leq t \leq \pi$.

H.~Busemann considered $2$-point homogeneous Busemann $G$-spaces (geodesic spaces) in 1942 \cite{Bus1942}. The latter can be described
as complete locally compact metric spaces $(M,d)$ such that any two points in
$M$ can be joined by a {\it shortest arc}, i.~e. an isometric embedding  $p:[0,d(x,y)]\rightarrow (M,d)$; shortest arcs are locally extendable; any shortest
arc has unique extension to a {\it geodesic}, i.~e. locally isometric mapping $\gamma: \mathbb{R}\rightarrow (M,d)$.

It is well known that every $2$-point homogeneous Riemannian manifold is isometric to Euclidean space or a Riemannian symmetric space of rank $1$, see, e.~g.
Theorem 8.12.8 and Corollary 8.12.9 in \cite{Wolf2011}.
In fact, this result holds for a more general situation, namely, for $2$-point homogeneous Busemann $G$-spaces. This can be easily deduced from papers  \cite{Wang1951},
\cite{Wang1952} by
H.C.~Wang for compact spaces and \cite{Tits1955} by J.~Tits for non-compact case.
Currently, there are various proofs of the above characterization of $2$-point homogeneous Riemannian manifolds. Especially important is very
short paper \cite{Sz1991} by Z.I.~Szab\'{o}, where he showed also that the above $2$-point homogeneous spaces are symmetric not using any classification,
thereby solving a problem of  Wolf from \cite{Wolf2011}.

H.~Busemann proved in \cite{Bus1956} that if for every two geodesics of a compact
Busemann $G$-space there exists an isometry of the space moving one of them to
the other (this condition is much weaker than $2$-point homogeneity of the space) then the space is isometric to one of compact Riemannian symmetric space of rank $1$.
V.N.~Berestovskii in~\cite{Ber2010} proved that non-compact Busemann $G$-space with the isometry group transitive on the family of its geodesics is isometric
to Euclidean space or non-compact Riemannian symmetric space of rank $1$, thus giving a positive answer to one Busemann problem.

As a corollary of quoted results, every $3$-point homogeneous Busemann $G$-space (in particular, Riemannian manifold) is isometric to Euclidean space, spherical
space, elliptic space, or hyperbolic space with appropriate constant sectional
curvature. This implies that in any such space every isometry of every two their subsets can be extended to an isometry of the entire space,
which is extremely different from the case with
the vertex sets of polytopes.

\vspace{15mm}

\end{document}